\documentclass[11pt]{article}

\usepackage{amsmath}
\usepackage{amssymb}
\usepackage{amsfonts}
\usepackage{amsthm}
\usepackage{array}
\usepackage{bm}
\usepackage[shortlabels]{enumitem}
\usepackage{stmaryrd}

\usepackage{amsxtra}
\usepackage{array}
\usepackage{mathrsfs}
\usepackage{slashed}
\usepackage{xcolor}
\usepackage{tikz-cd}
\usepackage{tkz-euclide} 
\usepackage{pgfplots}

\newcolumntype{C}[1]{>{\centering\hspace{0pt}}p{#1}}
\usepackage[margin=1in]{geometry}

\newcommand{\GL}{\mathrm{GL}}

\newcommand{\SO}{\mathrm{SO}}
\newcommand{\Spin}{\mathrm{Spin}}

\newcommand{\SU}{\mathrm{SU}}

\newcommand{\G}{\mathrm{G}}

\newcommand{\Id}{\mathrm{Id}}

\newcommand{\Cl}{\mathrm{Cl}}
\newcommand{\cl}{\mathrm{cl}}
\newcommand{\Dirac}{\slashed{D}}
\newcommand{\Spinor}{\slashed{S}}

\newcommand{\Sym}{\mathrm{Sym}}

\newcommand{\Z}{\mathbb{Z}}

\newcommand{\R}{\mathbb{R}}
\newcommand{\C}{\mathbb{C}}
\newcommand{\HH}{\mathbb{H}}

\newcommand{\CP}{\mathbb{CP}}
\newcommand{\HP}{\mathbb{HP}}

\newcommand{\vol}{\mathrm{vol}}

\newcommand{\SFF}{\mathrm{I\!I}}

\newtheorem{thm}{Theorem}[section]
\newtheorem{prop}[thm]{Proposition}
\newtheorem{lem}[thm]{Lemma}
\newtheorem{cor}[thm]{Corollary}

\theoremstyle{definition}
\newtheorem{defn}[thm]{Definition}

\newtheorem{rmk}[thm]{Remark}

\usepackage[nottoc]{tocbibind}
\numberwithin{equation}{section}

\usepackage[hidelinks]{hyperref}

\title{A Spinorial Hopf Differential for \\ Associative Submanifolds}
\author{Gavin Ball, Jesse Madnick}
\date{May 2023}

\newcommand{\Addresses}
{{  \bigskip
		\textsc{University of Wisconsin-Madison} \par\nopagebreak
		\textsc{Madison, WI} \par\nopagebreak
		\texttt{gball3@wisc.edu} \\
			
		\bigskip
		\textsc{University of Oregon} \par\nopagebreak
		\textsc{Eugene, OR}\par\nopagebreak
		\texttt{jmadnick@uoregon.edu}
}}

\begin{document}

\maketitle

\begin{abstract}
Given a CMC surface in $\R^3$, its traceless second fundamental form can be viewed as a holomorphic section called the Hopf differential.  By analogy, we show that for an associative submanifold of a 7-manifold $M^7$ with $\G_2$-structure, its traceless second fundamental form can be viewed as a twisted spinor.  Moreover, if $M$ is $\R^7$, $T^7$, or $S^7$ with the standard $\G_2$-structure, then this twisted spinor is harmonic.  Consequently, every non-totally-geodesic associative $3$-fold in $\R^7$, $T^7$, and $S^7$ admits non-vanishing harmonic twisted spinors.  Analogous results hold for special Lagrangians in $\R^6$ and $T^6$, coassociative $4$-folds in $\R^7$ and $T^7$, and Cayley $4$-folds in $\R^8$ and $T^8$.
\end{abstract}


\section{Introduction}

\indent \indent In minimal surface theory, Hopf's Theorem states that every immersed constant mean curvature (CMC) $2$-sphere in $\R^3$ is congruent to the round sphere.  Indeed, for any immersed orientable surface $u \colon \Sigma^2 \to \R^3$, its second fundamental form $\SFF$ can be split into traceless and trace parts as follows:
\begin{equation} \label{eq:SFFSplit}
\SFF = \SFF_0 + H\,\Id.
\end{equation}
Viewing $\Sigma$ as a Riemann surface, the complexification $\SFF_0^\C$ is a smooth section of $\Lambda^{1,0}(\Sigma) \otimes \Lambda^{1,0}(\Sigma)$ called the \emph{Hopf differential}, and is a holomorphic section if and only if $H$ is constant.  Consequently, if $u \colon \CP^1 \to \R^3$ is an immersed CMC $2$-sphere, then $\SFF_0^\C$ is a holomorphic section of $\Lambda^{1,0}(\CP^1)^{\otimes 2}$, a bundle which admits no non-vanishing holomorphic sections, so that $\SFF_0^\C = 0$, implying that $u \colon \CP^1 \to \R^3$ is congruent to the round sphere. \\
\indent Generalizations and analogues of the Hopf differential have since been discovered in various other settings.  For example, Abresh-Rosenberg \cite{AbreschRosenbergHopf} introduced a holomorphic quadratic form for CMC surfaces in $S^2 \times \R$ and $H^2 \times \R$.  As another example, Chern-Wolfson \cite{ChernWolfsonMoving} and Eells-Wood \cite{EellsWoodHarmonic} independently found a holomorphic cubic form for minimal surfaces in $\CP^n$ and $\HP^n$. \\

\indent In this note, we consider whether there exists a $\G_2$-geometric analogue of the above ideas.  That is, suppose $(M^7, \varphi)$ is a $7$-manifold with a $\G_2$-structure $\varphi$.  Given an associative submanifold $\Sigma^3 \subset M^7$, can its complexified traceless second fundamental form $\SFF_0^\C$ be viewed as a Hopf differential? \\
\indent On a $7$-manifold $M$, we recall that a \emph{$\G_2$-structure} is a $3$-form $\varphi \in \Omega^3(M)$ with the property that at each $x \in M$, the symmetric bilinear form $B_\varphi \colon \Sym^2(T_x^*M) \to \Lambda^7(T_x^*M)$ given by $B_\varphi(v,w) = (\iota_v\varphi) \wedge (\iota_w\varphi) \wedge \varphi$ is definite.  It is well-known that $M$ admits a $\G_2$-structure if and only if $M$ is orientable and spinnable.  An \emph{associative submanifold} of $(M, \varphi)$ is an oriented $3$-dimensional submanifold $\Sigma^3 \subset M^7$ satisfying the first-order PDE
$$\left.\varphi\right|_\Sigma = \vol_\Sigma.$$
If $\varphi$ is \emph{closed} (i.e., $d\varphi = 0$) or \textit{nearly-parallel} (i.e., $d\varphi = \lambda \ast\! \varphi$ for some constant $\lambda \neq 0$), then it is well-known that every associative $3$-fold in $(M, \varphi)$ is a minimal submanifold (i.e., $\SFF = \SFF_0$).  The converse was proven in \cite{BallMadnickExcept}. \\

\indent Taken literally, of course, the idea that $\SFF_0^\C$ could be a holomorphic section is absurd.  Indeed, since associative submanifolds are $3$-dimensional, they cannot be complex manifolds, and so one cannot speak of holomorphic vector bundles over $\Sigma$.  However, we will show that $\SFF_0^\C$ satisfies an analogue of the Cauchy-Riemann equations, namely the Dirac equation. \\
\indent In fact, there are two well-known relationships between associative submanifolds, on the one hand, and the Dirac equation, on the other.  First, suppose that $\Sigma \subset \R^7 = \R^3 \times \R^4$ is the graph of a smooth function $f \colon \R^3 \to \R^4$.  Identifying $\R^4 \simeq \HH$ and $\R^3 \simeq \text{Im}(\HH)$ with the quaternions and imaginary quaternions, Harvey and Lawson \cite[ $\S$IV.2.A.]{HarveyLawsonCalibrated} proved that $\Sigma$ is associative if and only if $f$ satisfies the PDE
$$\Dirac f = -\frac{\partial f}{\partial x_1} \times \frac{\partial f}{\partial x_2} \times \frac{\partial f}{\partial x_3},$$
where $\Dirac f = \frac{\partial f}{\partial x_1}i + \frac{\partial f}{\partial x_2}j + \frac{\partial f}{\partial x_3}k$ is the Dirac operator, and where $a \times b \times c = \star(a \wedge b \wedge c)$ is the standard triple cross product on $\R^4$. \\
\indent Second, suppose $(M^7, \varphi)$ is a $7$-manifold with a closed $\G_2$-structure, and let $\Sigma \subset M$ denote a compact associative submanifold (without boundary).  McLean proved \cite{McLeanDeformations} that under the natural isomorphism between the complexified normal bundle of $\Sigma$ and a certain twisted spinor bundle
$$N\Sigma \otimes_{\R} \C \cong \Spinor(T\Sigma) \otimes_{\C} \mathsf{V},$$
the infinitesimal associative deformations of $\Sigma$ correspond to those twisted spinors in the kernel of a Dirac operator $\Dirac$.  Analogues and extensions of McLean's Theorem are established, for example, in \cite{AkbulutSalur08}, \cite{LotayAsympConAssoc}, \cite{GayetWittDeformations}, \cite{GayetSmooth}, and \cite{MorenoSaEarpWeitzenbock}.  A related result is McLean's second variation of volume formula.  To recall it, suppose now that $\varphi$ is both closed and co-closed.  Then for any compactly-supported variation $(\Sigma_t)_{t \in (-\epsilon, \epsilon)}$ of an associative submanifold $\Sigma \subset M$, letting $\eta \in \Gamma(N\Sigma)$ denote the variation vector field, we have
$$\left. \frac{d^2}{dt^2}\right|_{t = 0} \text{Vol}(\Sigma_t) = \int_\Sigma \Vert \Dirac \eta \Vert^2\,\vol,$$
which again features the Dirac operator $\Dirac$ on the normal bundle.  This formula has been further studied in, for example, \cite{KawaiDeformationsHomog} and \cite{LeVanzuraMcLean}.

\indent Our main result establishes yet another link between associative geometry and spin geometry:

\begin{thm} \label{thm:Main} Let $(M, \varphi)$ be a $7$-manifold with a $\G_2$-structure.  Let $\Sigma^3 \subset M^7$ be an associative submanifold equipped with a choice of spin structure, let $\SFF_0 \in \Gamma(\Sym^2_0(T^*\Sigma) \otimes N\Sigma)$ denote the traceless part of its second fundamental form, and let $\SFF^\C_0$ be its complexification.
\begin{enumerate}[(a)]
\item There exists a complex vector bundle $\mathsf{E} \to \Sigma$ such that
$$(\Sym^2_0(T^*\Sigma) \otimes N\Sigma)^{\C} \cong \Spinor(T\Sigma) \otimes_{\C} \mathsf{E}.$$
Consequently, $\SFF_0^\C$ can be viewed as a twisted spinor.
\item If $M$ has constant curvature and $\varphi$ is torsion-free or nearly parallel (e.g., if $M = \R^7$, $T^7$, or $S^7$ with their standard $\mathrm{G}_2$-structures), then $\SFF_0^\C$ is a harmonic twisted spinor:
$$\Dirac \SFF_0^\C = 0.$$
Here, $\Dirac$ is the Dirac operator on $\Spinor(T\Sigma) \otimes \mathsf{E}$.
\end{enumerate}
\end{thm}

\begin{cor} \label{cor:Main} Let $M = \R^7$, $T^7$ or $S^7$ with its standard $\G_2$-structure.  If $\Sigma \subset M$ is an associative submanifold that is not totally geodesic, then $\Sigma$ admits a non-vanishing $\mathsf{E}$-twisted harmonic spinor for every spin structure on $\Sigma$. 
\end{cor} 

\indent Corollary \ref{cor:Main} is remarkable in view of the scarcity of harmonic twisted spinors on compact Riemannian $3$-manifolds $(\Sigma^3, g)$.  Indeed, since the elliptic operator $\Dirac$ has Fredholm index 0, the space $\mathrm{Ker}(\Dirac)$ is generically zero-dimensional.  Furthermore, note that a Dirac operator $\Dirac$ on a twisted spinor bundle $\Spinor \otimes \mathsf{E} \to \Sigma$ depends on the data $(g, \mathfrak{s}, h^{\mathsf{E}}, \nabla^{\mathsf{E}})$, where $(g, \mathfrak{s})$ are the metric and spin structure on $\Sigma$, and where $(h^{\mathsf{E}}, \nabla^{\mathsf{E}})$ are the Hermitian structure and connection on $\mathsf{E}$, respectively.  Angel \cite{AnghelGeneric} and Maier \cite{MaierGeneric} have shown that for generic choices of $\nabla^{\mathsf{E}}$, the Dirac operator $\Dirac$ admits no non-trivial harmonic twisted spinors at all. \\

\indent Our method also yields analogues of Theorem \ref{thm:Main} for other calibrated geometries, provided certain modifications are made.  For example, if $L^4 \subset Z^8$ is a Cayley $4$-fold in an $8$-manifold $(Z^8, \Phi)$ with a $\Spin(7)$-structure $\Phi \in \Omega^4(Z)$, then the bundle $(\Sym^2_0(T^*L) \otimes NL)^\C$ can be shown to be isomorphic to a twisted \emph{half}-spinor bundle $\Spinor^+ \otimes \mathsf{E}$.  Moreover, if $\Phi$ is flat, then $\SFF_0^\C$ is a harmonic twisted half-spinor. \\
\indent In general, further subtleties may arise.  For instance, suppose that $L^4$ is a coassociative $4$-fold in a $7$-manifold $(M^7, \varphi)$ with a $\G_2$-structure.  Then $(\Sym^2_0(T^*L) \otimes NL)^\C$ is not isomorphic to a twisted spinor or half-spinor bundle.  Nevertheless, there exists a \emph{subbundle} of $(\Sym^2_0(T^*L) \otimes NL)^\C$ isomorphic to a twisted half-spinor bundle, say $\Spinor^+ \otimes \mathsf{E}$.  If the $\G_2$-structure on $M$ is torsion-free, then $\SFF_0^\C$ is a section of that subbundle.  Moreover, if $\varphi$ is flat, then $\SFF_0^\C$ is a harmonic twisted spinor. \\

\noindent \textbf{Acknowledgements:} The second author thanks Chung-Jun Tsai for many conversations, and thanks Julius Baldauf for bringing helpful references on harmonic spinors to his attention.

\section{Proof of the Main Result}

\indent \indent Let $(M^7, \varphi)$ be a $7$-manifold with a $\G_2$-structure $\varphi \in \Omega^3(M)$, let $g_M$ denote the induced Riemannian metric on $M$, and let $\times \colon \Lambda^2(TM) \to TM$ denote the vector cross product, i.e.:
$$\varphi(u,v,w) = g_M(u \times v, w).$$
Let $u \colon \Sigma^3 \to M^7$ be an associative submanifold, and let $\SFF$ be its second fundamental form.  Decomposing $\SFF = \SFF_0 + H\,\text{Id}$ into traceless and trace parts as in (\ref{eq:SFFSplit}), the tensor $\SFF_0$ is a trace-free $N\Sigma$-valued quadratic form on $\Sigma$: 
$$\SFF_0 \in \Gamma(\Sym^2_0(T^*\Sigma) \otimes N\Sigma).$$
Here, we are using $N\Sigma$ to denote the normal bundle of $\Sigma$, and $\Sym^2_0(T^*\Sigma)$ to denote the vector bundle of trace-free quadratic forms on $\Sigma$. \\
\indent A oriented orthonormal frame $(e_1, \ldots, e_7)$ of $u^*(TM) \to \Sigma$ will be called an \emph{$\SO(4)$-frame} if it satisfies
\begin{align*}
\varphi(e_i, e_j, e_k) & = \epsilon_{ijk} & \text{span}(e_1, e_2, e_3) & = T\Sigma & \text{span}(e_4, e_5, e_5, e_7) & = N\Sigma,
\end{align*}
where here $\epsilon_{ijk} \in \{-1,0,1\}$ is the unique symbol that is antisymmetric in all three indices and satisfies 
\begin{align*}
& \epsilon_{123} = \epsilon_{145} = \epsilon_{167} = \epsilon_{246} = \epsilon_{275} = \epsilon_{374} = \epsilon_{365} = 1 \\
& \epsilon_{ijk} = 0 \ \text{ if }(i,j,k) \text{ is not a permutation of }(1,2,3),\, (1,4,5), \, \ldots,\, (3,6,5).
\end{align*}
We let $F \to \Sigma$ denote the principal $\SO(4)$-bundle whose fiber over $x \in \Sigma$ consists of the $\SO(4)$-frames at $x$. \\
\indent Since $\Sigma$ is an orientable $3$-manifold, it admits spin structures.  Once and for all, we choose a spin structure $\mathfrak{s}$ on $\Sigma$.  This choice determines a principal $\Spin(4)$-bundle $P_{\mathfrak{s}} \to \Sigma$ whose total space double covers $F$:
$$\begin{tikzcd}
\mathbb{Z}_2 \arrow[r]   & P_{\mathfrak{s}} \arrow[d] \\
\mathrm{SO}(4) \arrow[r] & F \arrow[d]                \\
                         & \Sigma                    
\end{tikzcd}$$
\indent In $\S$\ref{sec:SFFTwistedSpinor}, we will show that the complexification of $\Sym^2_0(T^*\Sigma) \otimes N\Sigma$ can be identified with a twisted spinor bundle $\Spinor(T\Sigma) \otimes \mathsf{E}$, thereby proving Theorem \ref{thm:Main}(a).  Geometrically, $\mathsf{E}$ is a vector bundle associated to the principal bundle $P_{\mathfrak{s}} \to \Sigma$ via a particular $\Spin(4)$-representation.  As such, we need some facts about the representation theories of $\Spin(4)$ and $\SO(4)$, to which we now turn.

\subsection{Representation Theory of $\Spin(4)$ and $\SO(4)$} \label{sec:RepTheory}

\indent \indent We now recall some basic facts about the representation theory of $\Spin(4)$ and $\SO(4)$.  First, recall that for each integer $p \geq 0$, the complex vector space $\Sym^p(\C^2) \simeq \C^{p+1}$ is an $\SU(2)$-representation in the following way: for $A \in \SU(2)$ and $f \in \Sym^p(\C^2)$, we define
$$(A \cdot f)(\mathbf{z}) := f(A^T\mathbf{z}).$$
Therefore, since $\Spin(4) \cong \SU(2) \times \SU(2)$, we have complex $\Spin(4)$-representations
$$U_{p,q} := \Sym^p(\C^2) \otimes_{\C} \Sym^q(\C^2) \simeq \C^{(p+1)(q+1)}$$
for all $p, q \in \Z$ with $p,q \geq 0$.  It is well-known that each $U_{p,q}$ is irreducible, and conversely that every complex irreducible $\Spin(4)$-representation is of this form.  For future use, we note that the Clebsch-Gordan formula yields the irreducible decompositions
\begin{align} \label{eq:ClebschGordan}
    U_{0,1} \otimes U_{p,q+1} & \cong U_{p,q+2} \oplus U_{p,q} \\
    U_{0,1} \otimes U_{p,0} & \cong U_{p,1}
\end{align}
\indent Next, recall that $\SO(4) \cong \Spin(4)/\Z_2$.  One can check that the $\Spin(4)$-action on $U_{p,q}$ descends to an $\SO(4)$-action if and only if $p \equiv q\, (\text{mod }2)$.  Moreover, in this case, the complex representation $U_{p,q}$ has a natural real structure $c \colon U_{p,q} \to U_{p,q}$, and we may consider its $+1$-eigenspace:
$$W_{p,q} = \left\{ f \in U_{p,q} \colon c(f) = f \right\} \simeq \R^{(p+1)(q+1)}$$
so that $W_{p,q}^\C \cong U_{p,q}$.  It is well-known that each $W_{p,q}$ is a real irreducible $\SO(4)$-module.  Conversely, every real irreducible $\SO(4)$-module is of the form $W_{p,q}$ for some non-negative $p,q \in \Z$ with $p \equiv q\,(\text{mod }2)$.

\subsection{The Second Fundamental Form as a Twisted Spinor} \label{sec:SFFTwistedSpinor}

\indent \indent We now use the above $\Spin(4)$- and $\SO(4)$-representations to construct two families of vector bundles over $\Sigma$.  First, for any $p,q \geq 0$, we consider the associated complex vector bundle
$$\mathsf{E}_{p,q} := P_{\mathfrak{s}} \times_{\mu} U_{p,q}.$$
Second, if $p \equiv q \ (\text{mod }2)$, we may also consider the associated \emph{real} vector bundle
$$\mathsf{V}_{p,q} := F \times_{\mu} W_{p,q}.$$
We reiterate that $\mathsf{V}_{p,q}$ is only defined when $p \equiv q \ (\text{mod }2)$.  In this case, we have
$$\mathsf{E}_{p,q} \cong \mathsf{V}_{p,q}^\C$$

\indent  Many vector bundles fundamental to the geometry of associatives $\Sigma \subset M$ are, in fact, of the form $\mathsf{E}_{p,q}$ or $\mathsf{V}_{p,q}$.  For example, the complex spinor bundle, the tangent bundle, and the normal bundle of $\Sigma$ are, respectively:
\begin{align*}
\Spinor := \Spinor(T\Sigma) & \cong \mathsf{E}_{0,1} & T\Sigma & \cong \mathsf{V}_{0,2} & N \Sigma & \cong \mathsf{V}_{1,1}.
\end{align*} 
Moreover, the bundle of trace-free quadratic forms with values in the normal bundle is:
$$\Sym^2_0(T^*\Sigma) \otimes N\Sigma \cong \mathsf{V}_{1,5} \oplus \mathsf{V}_{1,3}.$$

\begin{rmk}\label{rmk:secfundformspace} If $\varphi$ is torsion-free or nearly-parallel, then the second fundamental form of an associative $\Sigma \subset M$ satisfies the symmetry
$$\SFF(u \times v, w) = u \times \SFF(v,w) + \SFF(u,w) \times v.$$
Consequently, the traceless second fundamental form $\SFF_0$ is a section of the rank $12$ subbundle $\mathsf{V}_{1,5} \subset \Sym^2_0(T^*\Sigma) \otimes N\Sigma$.  For the time being, we do not assume that $\varphi$ is torsion-free nor nearly-parallel.
\end{rmk}

\indent In \cite{McLeanDeformations}, McLean observed that the complexified normal bundle of $\Sigma$ is isomorphic to a twisted spinor bundle.  To see this, we simply notice that
$$N\Sigma^\C \cong \mathsf{V}_{1,1}^\C \cong \mathsf{E}_{1,1} \cong \mathsf{E}_{0,1} \otimes \mathsf{E}_{1,0} \cong \Spinor \otimes \mathsf{E}_{1,0}.$$
In general, given a complex vector bundle over $\Sigma$, we wish to know whether it is isomorphic to a twisted spinor bundle.  The following lemma addresses this:

\begin{lem} \label{lem:TwistedSpinorCriterion} For any $p,q \geq 0$, we have
\begin{align*}
     \Spinor \otimes \mathsf{E}_{p,q+1}  & \cong \mathsf{E}_{p,q+2} \oplus \mathsf{E}_{p,q} \\
    \Spinor \otimes \mathsf{E}_{p,0} & \cong \mathsf{E}_{p,1}.
\end{align*}
\end{lem}

\begin{proof} We calculate, using (\ref{eq:ClebschGordan}), that
\begin{align*}
\Spinor \otimes \mathsf{E}_{p,q+1} \cong \mathsf{E}_{0,1} \otimes \mathsf{E}_{p,q+1} & = (P_{\mathfrak{s}} \times_{\mu} U_{0,1}) \otimes (P_{\mathfrak{s}} \times_{\mu} U_{p,q+1}) \\
& \cong P_{\mathfrak{s}} \times_{\mu} (U_{0,1} \otimes U_{p,q+1}) \\
& \cong P_{\mathfrak{s}} \times_{\mu} (U_{p,q+2} \oplus U_{p,q}) \\
& \cong (P_{\mathfrak{s}} \times_{\mu} U_{p,q+2}) \oplus (P_{\mathfrak{s}} \times_{\mu} U_{p,q}) \\
& = \mathsf{E}_{p,q+2} \oplus \mathsf{E}_{p,q}.
\end{align*}
The other isomorphism is proved similarly.
\end{proof}

\begin{cor} \label{cor:TwistedSpinorID} As vector bundles with structure group $\Spin(4)$, we have an isomorphism
$$(\Sym^2_0(T^*\Sigma) \otimes N\Sigma)^{\C} \cong \Spinor \otimes \mathsf{E}_{1,4}.$$
Thus, $\SFF^{\C}$ can be viewed as a twisted spinor.
\end{cor}

\begin{proof} By Lemma \ref{lem:TwistedSpinorCriterion} with $(p,q) = (1,3)$, we see that
$$\Spinor \otimes \mathsf{E}_{1,4} \cong \mathsf{E}_{1,5} \oplus \mathsf{E}_{1,3} \cong (\mathsf{V}_{1,5} \oplus \mathsf{V}_{1,3})^\C \cong \left( \Sym^2_0(T^*\Sigma) \otimes N\Sigma \right)^\C.$$
\end{proof}
\indent The isomorphism in the preceding corollary can be made explicit by tracing through the following alternative proof.  We observe that
\begin{align*}
\Spinor \otimes \mathsf{E}_{1,4} & \cong \mathsf{E}_{0,1} \otimes \mathsf{E}_{1,0} \otimes \Sym^4(\Spinor) \\
& \cong \mathsf{E}_{0,1} \otimes \mathsf{E}_{1,0} \otimes \Sym^2_0(\mathsf{E}_{0,2}) \\
& \cong (N\Sigma \otimes \Sym^2_0(T^*\Sigma))^\C.
\end{align*}
Concretely, for $\sigma \in \Spinor$ and $\xi \in \mathsf{E}_{1,0}$, we can view $\sigma \otimes \xi \in \mathsf{E}_{1,1} \cong N\Sigma^{\C}$ as a complex normal vector.  Moreover, for spinors $\rho_1, \rho_2 \in \Spinor$, we can view $\rho_1 \circ \rho_2 \in \Sym^2(\Spinor) \cong T^*\Sigma^\C$ as a complex cotangent vector.  Altogether, our isomorphism is described on decomposable elements as follows:
\begin{align*}
N\Sigma^\C \otimes \Sym^2_0(T^*\Sigma)^\C & \to \Spinor \otimes \mathsf{E}_{1,4} \\
(\sigma \otimes \xi) \otimes ( (\rho_1 \circ \rho_2) \circ (\rho_3 \circ \rho_4) ) & \mapsto \sigma \otimes (\xi \otimes \rho_1 \rho_2 \rho_3 \rho_4).
\end{align*}
Here, in writing $\rho_1\rho_2\rho_3\rho_4$, we mean the result of applying the natural map
\begin{align*}
\Sym^2_0(T^*\Sigma) \cong \Sym^2_0(\Sym^2(\Spinor)) & \to \Sym^4(\Spinor) \cong \mathsf{E}_{0,4} \\
q_1 \circ q_2 & \mapsto q_1 q_2
\end{align*}
to the trace-free quadratic form $(\rho_1 \circ \rho_2) \circ (\rho_3 \circ \rho_4)$.  

\subsection{The Twisted Spinor Bundle}

\indent \indent Having proved the first part of Theorem \ref{thm:Main} (as Corollary \ref{cor:TwistedSpinorID}), we now turn to the second part.  For this, we need to equip both $\Spinor \otimes \mathsf{E}_{1,4}$ and $(\Sym^2_0(T^*\Sigma) \otimes N\Sigma)^{\C}$ with suitable bundle metrics and connections.  In this section, we consider $\Spinor \otimes \mathsf{E}_{1,4}$.  In the following section, we examine  $(\Sym^2_0(T^*\Sigma) \otimes N\Sigma)^{\C}$ and draw conclusions.

\subsubsection{Structure on $\Spinor \otimes \mathsf{E}_{1,4}$}

\indent \indent The vector bundle $\mathsf{E}_{1,4} \to \Sigma$ is associated to the principal $\Spin(4)$-bundle $P_{\mathfrak{s}} \to \Sigma$ via the irreducible complex representation $\mu \colon \Spin(4) \to \GL(U_{1,4})$.  On $U_{1,4} = \C^2 \otimes \Sym^2(\C^4)$, we can exhibit a $\Spin(4)$-invariant Hermitian inner product via the tensor product of those on $\C^2$ and $\Sym^2(\C^4)$.  In fact, since $U_{1,4}$ is an irreducible representation of the compact Lie group $\Spin(4)$, this Hermitian inner product is the unique $\Spin(4)$-invariant one up to scale. \\
\indent Thus, we obtain an induced Hermitian bundle metric $\langle \cdot, \cdot \rangle$ on $\mathsf{E}_{1,4}$ as follows.  If we recast sections $\tau \in \Gamma(\mathsf{E}_{1,4})$ as equivariant functions $q \colon P_{\mathfrak{s}} \to U_{1,4}$, say $q = q_{\bar{i}} e_i$ for some unitary basis $(e_1, \ldots, e_{10})$ of $U_{1,4}$, then 
$$\langle \tau, \tau \rangle = q_{\bar{i}} \overline{q_{\bar{i}}}.$$
\indent Next, recall that the principal bundle $P_{\mathfrak{s}} \to \Sigma$ has a natural connection, namely the spin connection $\psi \in \Omega^1(P_{\mathfrak{s}}; \mathfrak{spin}(4))$.  Therefore, if $\tau \in \Gamma(\mathsf{E}_{1,4})$ is a local section, say $\tau = \tau_i \epsilon_i$ for some (local) functions $\tau_i \in C^\infty(\Sigma)$ and local frame $(\epsilon_1, \ldots, \epsilon_{10})$ of $\mathsf{E}_{1,4}$, then we may define the covariant derivative
$$\nabla^{\mathsf{E}} \colon \Gamma(\mathsf{E}_{1,4}) \to \Gamma( T^*\Sigma \otimes \mathsf{E}_{1,4} )$$
via
$$\nabla^{\mathsf{E}}\tau = d\tau_i \otimes \epsilon_i + \tau_i\, \mu_*(\psi)_{ij} \otimes \epsilon_j,$$
where $\mu_* \colon \mathfrak{spin}(4) \to \mathfrak{u}(U_{1,4})$ is the $\mathfrak{spin}(4)$-representation on $U_{1,4}$. \\

\indent Now, having equipped $\mathsf{E}_{1,4}$ with a Hermitian bundle metric and connection $\nabla^{\mathsf{E}}$ as above, a completely analogous procedure endows the spinor bundle $\Spinor \cong \mathsf{E}_{0,1}$ with a Hermitian bundle metric and connection $\nabla^{\Spinor}$.  Finally, we may equip $\Spinor \otimes \mathsf{E}_{1,4}$ with its tensor product metric
$$\langle \sigma_1 \otimes \tau_1, \sigma_2 \otimes \tau_2 \rangle := \langle \sigma_1, \sigma_2 \rangle \langle \tau_1, \tau_2 \rangle,$$
and tensor product connection
\begin{align*}
\nabla \colon \Gamma(\Spinor \otimes \mathsf{E}_{1,4}) & \to \Gamma( (T^*\Sigma)^\C \otimes \Spinor \otimes \mathsf{E}_{1,4}) \\
\nabla(\sigma \otimes \tau) & = (\nabla^{\Spinor} \sigma) \otimes \tau + \sigma \otimes \nabla^{\mathsf{E}}\tau.
\end{align*}

\subsubsection{The Dirac Operator on $\Spinor \otimes \mathsf{E}_{1,4}$}

\indent \indent Let us now recall some relevant definitions, following \cite[$\S$II.5]{LawsonMichelsonSpin} and \cite[$\S$12.1]{NicolaescuLectures}.

\begin{defn} Let $E \to (\Sigma,g)$ be a complex vector bundle over a Riemannian manifold $(\Sigma, g)$ with Levi-Civita connection $\nabla^{\mathrm{LC}}$.  A \emph{Dirac structure on $E$} is a triple $(\langle \cdot, \cdot \rangle, \nabla, \cl)$ consisting of a Hermitian bundle metric $\langle \cdot, \cdot \rangle$, a connection $\nabla$, and a $\Cl(T^*\Sigma)^\C$-module structure on $E$, denoted
\begin{align*}
\cl \colon T^*\Sigma^\C \otimes E & \to E \\
\cl(\alpha \otimes \sigma) & = \alpha \cdot \sigma,
\end{align*}
that are pairwise compatible in the following sense:
\begin{enumerate}[(1)]
\item $\nabla$ is $\langle \cdot, \cdot \rangle$-compatible.
\item $\langle \alpha \cdot \sigma_1, \alpha \cdot \sigma_2 \rangle = |\alpha|^2 \langle \sigma_1, \sigma_2 \rangle$ for all $\alpha \in T^*\Sigma^\C$ and $\sigma \in E$.
\item $\nabla(\alpha \cdot \sigma) = (\nabla^{\mathrm{LC}} \alpha) \cdot \sigma + \alpha \cdot \nabla \sigma$ for all $\alpha \in \Gamma(T^*\Sigma^\C)$ and $\sigma \in \Gamma(E)$.
\end{enumerate}
\indent A \emph{Dirac bundle} is a vector bundle $E \to \Sigma$ equipped with a Dirac structure.  In this case, its \emph{Dirac operator} is the composition
$$\Dirac = \mathrm{cl} \circ \nabla \colon \Gamma(E) \to \Gamma(E).$$
Diagrammatically,
$$\Dirac \colon \Gamma(E) \xrightarrow{\nabla} \Gamma(T^*\Sigma^\C \otimes E) \xrightarrow{\cl} \Gamma(E).$$
\end{defn}

\indent We now equip $\Spinor \otimes \mathsf{E}_{1,4}$ with its standard Dirac structure.  For this, recall that $\Spinor \otimes \mathsf{E}_{1,4}$ already carries the tensor product metric $\langle \cdot, \cdot \rangle$ and tensor product connection $\nabla$.  Now, for $\alpha \in T^*\Sigma^\C$, $\sigma \in \Spinor$, and $\tau \in \mathsf{E}_{1,4}$, we define
\begin{align*}
    \text{cl} \colon T^*\Sigma^\C \otimes ( \Spinor \otimes \mathsf{E}_{1,4}) & \to  \Spinor \otimes \mathsf{E}_{1,4} \\
\alpha \cdot (\sigma \otimes \tau) & := (\alpha \cdot \sigma) \otimes \tau.
\end{align*}
It is then a standard fact \cite[$\S$II.5]{LawsonMichelsonSpin}, \cite[p. 643]{NicolaescuLectures} that the triple $(\langle \cdot, \cdot  \rangle, \nabla, \mathrm{cl})$ is a Dirac structure on $\Spinor \otimes \mathsf{E}_{1,4}$.  In particular, $\Spinor \otimes \mathsf{E}_{1,4}$ admits a Dirac operator 
$$\Dirac \colon \Gamma( \Spinor \otimes \mathsf{E}_{1,4} ) \to \Gamma( \Spinor \otimes \mathsf{E}_{1,4} ).$$

\subsubsection{A Condition for Harmonicity}

\indent \indent We now give a sufficient condition for a twisted spinor $\gamma \in \Gamma(\Spinor \otimes \mathsf{E}_{1,4})$ to satisfy $\Dirac \gamma = 0$.  For this, we observe that the isomorphism
\begin{align*}
(T^*\Sigma)^\C \otimes \Spinor \otimes \mathsf{E}_{1,4} & \cong \mathsf{E}_{0,2} \otimes (\mathsf{E}_{1,5} \oplus \mathsf{E}_{1,3}) \\
& \cong \mathsf{E}_{1,7} \oplus 2\mathsf{E}_{1,5} \oplus 2\mathsf{E}_{1,3} \oplus \mathsf{E}_{1,1}
\end{align*}
allows us to regard the Dirac operator as a composition
$$\Dirac \colon \Gamma(\Spinor \otimes \mathsf{E}_{1,4}) \xrightarrow{\nabla} \Gamma(\mathsf{E}_{1,7} \oplus 2\mathsf{E}_{1,5} \oplus 2\mathsf{E}_{1,3} \oplus \mathsf{E}_{1,1}) \xrightarrow{\cl} \Gamma(\Spinor \otimes \mathsf{E}_{1,4}).$$
The equivariance of $\mathrm{cl}$, together with Schur's Lemma, now yields the following: \\

\begin{prop} \label{prop:SufficientHarmonic} Let $\gamma \in \Gamma(\Spinor \otimes \mathsf{E}_{1,4})$ be a twisted spinor.  If $\nabla \gamma \in \Gamma(\mathsf{E}_{1,7} \oplus \mathsf{E}_{1,1})$, then $\Dirac \gamma = 0$.
\end{prop}

\begin{proof} Consider the bundle map given by the Clifford action
$$\cl \colon (T^*\Sigma)^\C \otimes \Spinor \otimes \mathsf{E}_{1,4} \to \Spinor \otimes \mathsf{E}_{1,4},$$
and recall the isomorphisms
\begin{align*}
(T^*\Sigma)^\C \otimes \Spinor \otimes \mathsf{E}_{1,4} & \cong \mathsf{E}_{1,7} \oplus 2\mathsf{E}_{1,5} \oplus 2\mathsf{E}_{1,3} \oplus \mathsf{E}_{1,1} & \Spinor \otimes \mathsf{E}_{1,4} & \cong \mathsf{E}_{1,5} \oplus \mathsf{E}_{1,3}.
\end{align*}
The map $\mathrm{cl}$ is $\Spin(4)$-equivariant.  Therefore, by Schur's Lemma, if $\beta \in \mathsf{E}_{1,7} \oplus \mathsf{E}_{1,1}$, then $\cl(\beta) = 0$.  Taking $\beta = \nabla \gamma$ proves the result.
\end{proof}

\subsection{Proof of Theorem \ref{thm:Main}(b)}

\indent \indent We now return to the bundle $(\Sym^2_0(T^*\Sigma) \otimes N\Sigma)^{\C}$.  As both $T\Sigma$ and $N\Sigma$ carry natural bundle metrics, the real vector bundle $\Sym^2_0(T^*\Sigma) \otimes N\Sigma$ inherits a natural bundle metric.  Extending by sesquilinearity yields a Hermitian bundle metric on $(\Sym^2_0(T^*\Sigma) \otimes N\Sigma)^{\C}$.  The isomorphism
$$F \colon  (\Sym^2_0(T^*\Sigma) \otimes N\Sigma)^{\C}  \xrightarrow{\cong} \Spinor \otimes \mathsf{E}_{1,4}$$
of Corollary \ref{cor:TwistedSpinorID} is then a Hermitian bundle isometry. \\
\indent Now, the Levi-Civita connection on $M$ induces tangential and normal connections on $T\Sigma$ and $N\Sigma$, which in turn induce a connection on $\Sym^2_0(T^*\Sigma) \otimes N\Sigma$.  After extending by $\C$-linearity, we let $\nabla^{\mathrm{LC}}$ denote the resulting connection on $(\Sym^2_0(T^*\Sigma) \otimes N\Sigma)^{\C}$.  One can check that the isomorphism $F$ also intertwines the connections, in the sense that for any $\sigma \in \Gamma((\Sym^2_0(T^*\Sigma) \otimes N\Sigma)^{\C})$,
$$(\Id \otimes F)(\nabla^{\mathrm{LC}}\sigma) = \nabla(F(\sigma)).$$
As such, we will suppress $F$ from the notation.  We are now in a position to prove:

\begin{prop} \label{prop:IntertwineConnections} Suppose that $M$ has constant sectional curvature and $\varphi$ is torsion-free or nearly parallel.  Then $\Dirac(\SFF^{\C}) = 0$.
\end{prop}

\begin{proof} By Remark \ref{rmk:secfundformspace}, $\SFF$ is a section of $\mathsf{V}_{1,5} \subset N \Sigma \otimes \mathrm{Sym}^2_0 (T^* \Sigma)$. Therefore, $\nabla^{\mathrm{LC}} (\SFF)$ is a section of the subbundle $\mathsf{V}_{1,5} \otimes T^* \Sigma \subset N \Sigma \otimes \mathrm{Sym}^2_0 (T^* \Sigma) \otimes T^* \Sigma.$ Abstractly, this subbundle is isomorphic to $\mathsf{V}_{1,7} \oplus \mathsf{V}_{1,5} \oplus \mathsf{V}_{1,3}.$ On the other hand, since $M$ has constant sectional curvature, the Codazzi equation implies that $\nabla^{\mathrm{LC}}(\SFF)$ is a section of $N \Sigma \otimes \mathrm{Sym}^3_0 (T^* \Sigma)$.  Direct calculation yields
	\begin{equation*}
		( \mathsf{V}_{1,5} \otimes T^* \Sigma ) \cap (N \Sigma \otimes \mathrm{Sym}^3_0 (T^* \Sigma)) \cong \mathsf{V}_{1,7},
	\end{equation*}
where the intersection is taken inside $N \Sigma \otimes T^* \Sigma \otimes T^* \Sigma \otimes T^* \Sigma.$ Therefore, the complexification $\nabla(\SFF^\C)$ is  section of  $\mathsf{E}_{1,7}$.  Applying Proposition \ref{prop:SufficientHarmonic}, we conclude that $\Dirac(\SFF^{\C}) = 0$.
\end{proof}

\bibliographystyle{plain}
\bibliography{SHDRef}

\Addresses
 
\end{document}